\documentclass[12pt,a4paper]{amsart}
\usepackage{amsmath}
\usepackage{drpack}
\usepackage[english]{babel}
\usepackage{verbatim}
\usepackage[shortlabels]{enumitem}
\usepackage{color}
\usepackage{tikz-cd}

\newtheoremstyle{mio}%
{}{} 
{\itshape}{} 
{\bfseries}{.}{ } 
{#1 #2\thmnote{~\mdseries(#3)}} 
\theoremstyle{mio}
\newtheorem{teor}{Theorem}[section]
\newtheorem{cor}[teor]{Corollary}
\newtheorem{prop}[teor]{Proposition}
\newtheorem{lemma}[teor]{Lemma}
\newtheorem{defin}[teor]{Definition}

\newtheoremstyle{definition2}%
{}{} 
{}{} 
{\bfseries}{.}{ } 
{#1 #2\thmnote{\mdseries~ #3}} 
\theoremstyle{definition2}
\newtheorem{ex}[teor]{Example}
\newtheorem{oss}[teor]{Remark}

\newcommand{\inverse}{\mathrm{inv}}
\newcommand{\inssemistar}{\mathrm{SStar}}
\newcommand{\insstable}{\mathrm{SStar_{st}}}
\newcommand{\inssubmod}{\mathbf{F}}
\newcommand{\ortog}{\perp}
\newcommand{\isolated}{\mathcal{I}}
\newcommand{\limitp}{\mathcal{D}}
\newcommand{\insmod}{\mathrm{Mod}}
\newcommand{\inslength}{\mathcal{L}}
\newcommand{\inslengthsing}{\inslength_{\mathrm{sing}}}
\newcommand{\cons}{\mathrm{cons}}
\newcommand{\V}{\mathcal{V}}
\newcommand{\D}{\mathcal{D}}
\newcommand{\B}{\mathcal{B}}

\title{The derived sequence of a pre-Jaffard family}
\author{Dario Spirito}
\date{\today}
\address{Dipartimento di Matematica ``Tullio Levi-Civita'', Universit\`a degli Studi di Padova, Padova, Italy}
\email{spirito@math.unipd.it}

\begin{document}

\begin{abstract}
We introduce the concept of \emph{pre-Jaffard family}, a generalization of Jaffard families obtained by substituting the locally finite hypothesis with a much weaker compactness hypothesis. From any such family, we construct a sequence of overrings of the starting domain that allows to decompose stable semistar operations and singular length functions in more cases than what is allowed by Jaffard families. We also apply the concept to one-dimensional domains, unifying the treatment of sharp and dull degree of a Pr\"ufer domain.
\end{abstract}

\maketitle

\section{Introduction}
A \emph{Jaffard family} of an integral domain $D$ is a family of flat overrings of $D$ that satisfy some strong independence property, while simultaneously respecting the structure of the whole ring (see Definition \ref{defin:Jaffard} for a precise definition). This notion allows to extend several results of $h$-local domains to more general rings; in particular, it was used to extend factorization properties from domains of Dedekind type to a wider class of domains \cite[Chapter 6]{fontana_factoring}. Jaffard families were subsequently used to factorize the set of star operations \cite{starloc,starloc2} and the set of length functions \cite{length-funct} on an integral domain as the product of the analogous sets on the members of a Jaffard family.

In this paper, we introduce two generalizations of Jaffard families, namely \emph{weak Jaffard families} and \emph{pre-Jaffard families}.

Weak Jaffard families (see Section \ref{sect:weakJaffard}) are very similar to Jaffard families, with the exception that we allow for a single member of the family to behave ``badly''.

Pre-Jaffard families (see Section \ref{sect:preJaffard}), on the other hand, need to satisfy weaker hypothesis, but for these reason are much more common; for example, the set of localizations at the maximal ideals of a domain of dimension $1$ is always a pre-Jaffard family. We show in Section \ref{sect:derived} how every pre-Jaffard family $\Theta$ generates a sequence of weak Jaffard families; this sequence is constructed very similarly to the sequence of derived sets of a topological space, and for this reason we call it the \emph{derived sequence} of $\Theta$. Indeed, when the dimension of the base ring $D$ is $1$, the members of the derived sequence of $\Theta$ correspond naturally to the member of the sequence of derived sets\ of the maximal space of $D$, endowed with the inverse topology. In particular, our construction is a generalization of the study of sharp and dull primes tackled in \cite[Section 6]{HK-Olb-Re} for one-dimensional Pr\"ufer domains; moreover, our terminology symmetrizes some of their results by unifying the concept of sharp and dull degree of a one-dimensional domain into the concept of \emph{Jaffard degree} of a pre-Jaffard family. See Section \ref{sect:dim1} for the discussion.

In Section \ref{sect:stable}, we apply weak Jaffard families to the study of singular length functions and of star operations: we show that, given a pre-Jaffard family, we can factorize their set through the derived sequence (Theorem \ref{teor:stagJ-stable}) allowing a wide generalization of the results on Jaffard families and of \cite[Example 6.9]{length-funct}.

\section{Preliminaries}
Throughout the paper, all rings will be commutative, unitary and without zero-divisors, i.e., integral domains; we denote such a ring by $D$, and we will always use $K$ to denote its quotient field.

We use $\Spec(D)$ and $\Max(D)$, respectively, to denote the spectrum and the maximal spectrum of $D$, and we denote by $\D(I)$ and $\V(I)$, respectively, the open and the closed set of $\Spec(D)$ associated to the ideal $I$. The \emph{inverse topology} on $\Spec(D)$ is the topology generated by the $\V(I)$, as $I$ ranges among the finitely generated ideals of $D$. We denote by $\Delta^\inverse$ a subset $\Delta\subseteq\Spec(D)$ endowed with the inverse topology.

The \emph{constructible topology} is the topology generated by the $\D(I)$ and the $\V(J)$ (as $I$ ranges among all ideals and $J$ among all finitely generated ideals); the constructible topology is still compact, but it is also Hausdorff. We denote by $\Delta^\cons$ a subset $\Delta\subseteq\Spec(D)$ endowed with the constructible topology. See \cite[Chapter 1]{spectralspaces-libro} for the construction and properties of the inverse and the constructible topology.

\subsection{Overrings}\label{sect:overrings}
An \emph{overring} of $D$ is a ring $T$ such that $D\subseteq T\subseteq K$; the set of all overrings of $D$ is denoted by $\Over(D)$. This set can be endowed with a topology (called the \emph{Zariski topology}) by taking as a subbasis the family of sets
\begin{equation*}
\B(x_1,\ldots,x_n):=\{T\in\Over(D)\mid x_1,\ldots,x_n\in T\},
\end{equation*} 
as $x_1,\ldots,x_n$ range in $K$. Under this topology, $\Over(D)$ is a compact space that is not Hausdorff, and furthermore it is a \emph{spectral space} in the sense of Hochster \cite{hochster_spectral}, i.e., there is a ring $A$ (in general not determined explicitly) such that $\Spec(A)$ (endowed with the Zariski topology) is homeomorphic to $\Over(D)$ (see e.g. \cite[Proposition 3.5]{finocchiaro-ultrafiltri}). The name ``Zariski topology'' is also due to the fact that the localization map
\begin{equation*}
\begin{aligned}
\lambda\colon\Spec(D) & \longrightarrow\Over(D),\\
P & \longmapsto D_P
\end{aligned}
\end{equation*}
is continuous and, indeed, a topological embedding \cite[Lemma 2.4]{dobbs_fedder_fontana}.

The \emph{closure under generizations} of a set $\Theta\subseteq\Over(D)$ is
\begin{equation*}
\Theta^\uparrow:=\{T\in\Over(D)\mid T\supseteq S\text{~for some~}S\in\Theta\};
\end{equation*}
a family $\Theta$ is \emph{closed by generizations} if $\Theta=\Theta^\uparrow$. The family of all sets that are closed by generizations and compact with respect to the Zariski topology is the family of closed sets of a topology, called the \emph{inverse topology} of $\Over(D)$; equivalently, the inverse topology is the topology generated by the complements of the sets $\B(x_1,\ldots,x_n)$. 

The \emph{constructible topology} on $\Over(D)$ is the topology generated by both the sets $\B(x_1,\ldots,x_n)$ and its complements. The space of all overrings, under both the inverse and the constructible topology, is again compact and a spectral space; moreover, under the constructible topology it is Hausdorff. Every set that is closed in the constructible topology is compact with respect to the Zariski topology. 

\subsection{Isolated points}
Let $X$ be a topological space. A point $x\in X$ is \emph{isolated} if $\{x\}$ is an open set; we denote the set of isolated points of $X$ by $\isolated(X)$. The set of non-isolated (i.e., limit) points is called the \emph{derived set} of $X$ and is denoted by by $\limitp(X)$.

We set $\limitp^0(X):=X$ and, for every ordinal $\alpha$, we define:
\begin{equation*}
\limitp^\alpha(X):=\begin{cases}
\limitp(\limitp^\gamma(X)) & \text{if~}\alpha=\gamma+1\text{~is a successor ordinal},\\
\bigcap_{\beta<\alpha}\limitp^\beta(X) & \text{if~}\alpha\text{~is a limit ordinal}.
\end{cases}
\end{equation*}
The set $\limitp^\alpha(X)$ is called the \emph{$\alpha$-th Cantor-Bendixson derivative} of $X$, and the smallest ordinal $\alpha$ such that $\limitp^\alpha(X)=\limitp^{\alpha+1}(X)$ is called the \emph{Cantor-Bendixson rank} of $X$. If $\limitp^\alpha(X)=\emptyset$ for some ordinal $\alpha$, the space $X$ is said to be \emph{scattered}; equivalently, $X$ is scattered if and only if every nonempty subspace has an isolated point.

\subsection{Semistar operations and length functions}
Let $D$ be a domain and let $\inssubmod_D(K)$ be the set of $D$-submodules of $K$. A \emph{semistar operation} on $D$ is a map $\star:\inssubmod_D(K)\longrightarrow\inssubmod_D(K)$, $I\mapsto I^\star$, such that, for every $I,K\in\inssubmod_D(K)$ and every $x\in K$:
\begin{itemize}
\item $I\subseteq I^\star$;
\item if $I\subseteq J$, then $I^\star\subseteq J^\star$;
\item $(I^\star)^\star)=I^\star$;
\item $x\cdot I^\star=(xI)^\star$.
\end{itemize}
If $(I\cap J)^\star=I^\star\cap J^\star$ for every $I,J$, we say that $\star$ is \emph{stable}. We denote the sets of semistar operations and of stable semistar operations, respectively, by $\inssemistar(D)$ and $\insstable(D)$. These two sets have a partial order, given by $\star_1\leq\star_2$ if and only if $I^{\star_1}\subseteq I^{\star_2}$ for every ideal $I$; the infimum of a family $\Delta$ of semistar operations is the map $\sharp:I\mapsto\bigcap\{I^\star\mid\star\in\Delta\}$. If $\Delta\subseteq\insstable(D)$, then also $\sharp$ is stable.

Let $\insmod(D)$ be the category of $D$-modules. A \emph{length function} on $D$ is a function $\ell:\insmod(D)\longrightarrow\insR^{\geq 0}\cup\{\infty\}$ such that:
\begin{itemize}
\item $\ell(0)=0$;
\item if $0\longrightarrow N\longrightarrow M\longrightarrow P\longrightarrow 0$ is an exact sequence, then $\ell(M)=\ell(P)+\ell(N)$;
\item for every module $M$, we have $\ell(M)=\sup\{\ell(N)\mid N$ is a finitely generated submodule of $M\}$.
\end{itemize}
The sum of a family $\Lambda$ of length functions is defined as the map such that
\begin{equation*}
\left(\sum_{\ell\in\Lambda}\ell\right)(M)=\sup\{\ell_1(M)+\cdots+\ell_n(M)\},
\end{equation*}
as $\{\ell_1,\ldots,\ell_n\}$ ranges among the finite subsets of $\Lambda$.

If $T$ is a flat overring of $D$, then we can associate to any length function $\ell$ on $D$ a length function $\ell^D$ on $T$ by restriction of scalars, i.e., setting $\ell^D(M):=\ell(M)$ for all $T$-modules $M$. Moreover, we can defined a new length function $\ell\otimes T$ on $D$ by setting
\begin{equation*}
(\ell\otimes T)(M):=\ell(M\otimes T).
\end{equation*}
for all $D$-modules $M$.

By \cite[Theorem 6.5]{length-funct} and the subsequent discussion, there is a bijection between the set $\inslengthsing(D)$ of length functions such that $\ell(M)\in\{0,+\infty\}$ for all $M\in\insmod(D)$ and the set $\insstable(D)$ of stable semistar operations on $D$, and by \cite[Proposition 6.6]{length-funct} the infimum of a family of stable operations correspond to the sum of the corresponding length functions. Moreover, the passage from a length function $\ell$ on a flat overring $T$ to $\ell^D$ correspond to the passage from the a stable operation $\star$ on $T$ to the closure $I\mapsto(IT)^\star$ on $D$.

\section{Jaffard overrings}\label{sect:jaffard}
In this paper we will mostly use families consisting of \emph{flat overrings}, i.e., overrings of a domain $D$ that are flat when considered as $D$-modules. However, in many case we will define rings by intersecting localizations; thus we need the following definition.
\begin{defin}
An overring $T$ of $D$ is a \emph{sublocalization} of $D$ if there is a set $\Delta\subseteq\Spec(D)$ such that $T=\bigcap\{D_P\mid P\in\Delta\}$.

If $T$ is a sublocalization of $D$, we set:
\begin{itemize}
\item $\sigma(T):=\{Q\cap D\mid Q\in\Spec(T)\}$;
\item $\Sigma(T):=\{P\in\Spec(D)\mid T\subseteq D_P\}$;
\item $T^\ortog:=\bigcap\{D_P\mid P=(0)\text{~or~}P\in\Spec(D)\setminus\Sigma(T)\}$.
\end{itemize}
\end{defin}

Note that, by definition, $T^\ortog$ is a sublocalization too, and thus it makes sense to consider $\sigma(T^\ortog)$ and $\Sigma(T^\ortog)$.

\begin{lemma}\label{lemma:Sigmacup}
For every sublocalization $T$ of $D$, we have $\Sigma(T)\cup\Sigma(T^\ortog)=\Spec(D)$.
\end{lemma}
\begin{proof}
If $P\notin\Sigma(T)$, then by definition $T^\ortog\subseteq D_P$, and thus $P\in\Sigma(T^\ortog)$.
\end{proof}

Every flat overring is a sublocalization \cite[Corollary to Theorem 2]{richamn_generalized-qr}, but the converse is not true (see \cite{well-centered} and \cite[Example 6.3]{localizzazioni}). We can characterize when a sublocalization is flat.
\begin{lemma}\label{lemma:Sigmaiota}
Let $T$ be a sublocalization of $D$. Then, $T$ is flat over $D$ if and only if $\Sigma(T)=\sigma(T)$.
\end{lemma}
\begin{proof}
The containment $\Sigma(T)\subseteq\sigma(T)$ holds for every sublocalization. 

If $T$ is flat and $P\in\sigma(T)$, then $PT\neq T$, and by \cite[Theorem 1]{richamn_generalized-qr} we have $T\subseteq D_P$, i.e., $P\in\Sigma(T)$. Conversely, if $\sigma(T)=\Sigma(T)$ and $PT\neq T$, let $Q$ be a prime ideal of $T$ above $PT$: then, $P':=Q\cap D\in\sigma(T)$ and thus $T\subseteq D_{P'}$. Hence, $PT\cap D\subseteq PD_{P'}\cap D=P$, and so $P'=P$, and in particular $T\subseteq D_P$. Again by \cite[Theorem 1]{richamn_generalized-qr}, $T$ is flat.
\end{proof}

\begin{lemma}\label{lemma:sigmacap}
Let $A$ be a flat overring and $B$ a sublocalization of $D$. Then, $AB=K$ if and only if $\sigma(A)\cap\sigma(B)=\{(0)\}$.
\end{lemma}
\begin{proof}
Suppose that $AB=K$, and let $P\in\sigma(A)\cap\sigma(B)$. Since $A$ is flat, $P\in\Sigma(A)$ and so $A\subseteq D_P$; on the other hand, if $Q$ is a prime ideal of $B$ above $P$, then $D_P\subseteq B_Q$. Hence, $K=ABD_P=(AD_P)(BD_P)\subseteq D_PB_Q=B_Q$. It follows that $Q=(0)$ and so $P=(0)$ too.

Conversely, if $\sigma(A)\cap\sigma(B)=\{(0)\}$ then the claim follows from \cite[Lemma 6.2.1]{fontana_factoring}.
\end{proof}

\begin{defin}
Let $D$ be an integral domain with quotient field $K$ and let $\Theta$ be a family of overrings of $D$. We say that $\Theta$ is:
\begin{itemize}
\item \emph{complete} if, for every ideal $I$ of $D$, we have $I=\bigcap\{IT\mid T\in\Theta\}$;
\item \emph{independent} if, for every $A,B\in\Theta$, $\sigma(A)\cap\sigma(B)=\{(0)\}$;
\item \emph{strongly independent} if, for every $A\in\Theta$, we have
\begin{equation*}
A\cdot\left(\bigcap_{\substack{B\in\Theta\\ B\neq A}}B\right)=K
\end{equation*}
\item \emph{locally finite} if every $x\in K$, $x\neq 0$, is a nonunit in only finitely many members of $\Theta$.
\end{itemize}
\end{defin}

By Lemma \ref{lemma:sigmacap}, if $\Theta$ is a family of flat subsets, then $\Theta$ is independent if and only if $AB=K$ for every $A\neq B$ in $\Theta$; in particular, a strongly independent set of flat overrings is always independent. We can prove when the converse happens.
\begin{prop}\label{prop:stronglyindep}
Let $\Theta$ be a complete and independent set of flat overring of $D$. Then, $\Theta$ is strongly independent if and only if $\Theta$ is locally finite.
\end{prop}
\begin{proof}
If $\Theta$ is locally finite, the claim follows from \cite[Theorem 6.3.1(4)]{fontana_factoring} (see below for the definitions used in the reference). Suppose that $\Theta$ is strongly independent but not locally finite: then, there is a nonzero $x\in D$ such that $xT\neq T$ for an infinite family $\Theta'\subseteq\Theta$. Hence, for each $T\in\Theta'$ there is a prime ideal $P_T\in\Spec(D)$ such that $x\in P_T$ and $P_TT\neq T$; let $\Lambda$ be the family of such ideals. Then, $\Lambda$ is an infinite subset of the compact space $\Spec(D)^\cons$, and thus it has a limit point $Q$; furthermore, $\Lambda$ is contained in the clopen set $\V(x)$ of $\Spec(D)^\cons$ and thus also $Q\in\V(x)$, i.e., $x\in Q$; in particular, $Q\neq(0)$.

Since $\Theta$ is complete and independent, there is a unique $S\in\Theta$ such that $QS\neq S$; let $A:=\bigcap\{T\in\Theta\mid T\neq S\}$. Then, $A$ is a sublocalization of $D$; by \cite[Theorem 1]{richamn_generalized-qr}, $\sigma(A)$ is the image of $\Spec(A)$ under the restriction map $Z\mapsto Z\cap D$, and thus $\sigma(A)$ is a closed set in the constructible topology. Moreover, $\sigma(A)$ contains all the elements of $\Lambda$ except one (the ideal $P_S$), and thus it must contain also the limit point $Q$ of $\Lambda\setminus\{P_S\}$. Therefore, $D_QA\neq K$. However, $S\subseteq D_Q$ since $S$ is flat; therefore, $AS\subseteq D_QS\neq K$. This contradicts the fact that $\Theta$ is strongly independent: hence $\Theta$ must be locally finite.
\end{proof}

Note that the above result does not hold without the hypothesis that $\Theta$ is complete: see Example \ref{ex:almded} below.

Families satisfying the hypothesis of the previous proposition have their own name.
\begin{defin}\label{defin:Jaffard}
Let $\Theta$ be a family of overrings of $D$. We say that $\Theta$ is a \emph{Jaffard family} of $D$ if:
\begin{itemize}
\item either $K\notin\Theta$ or $\Theta=\{K\}$;
\item every $T\in\Theta$ is flat;
\item $\Theta$ is complete;
\item $\Theta$ is independent;
\item $\Theta$ is locally finite.
\end{itemize}
We say that an overring $T$ of $D$ is a \emph{Jaffard overring} if it belongs to a Jaffard family of $D$.
\end{defin}

\begin{oss}
~\begin{enumerate}
\item Definition \ref{defin:Jaffard} is not the original one of a Jaffard family, but it is the one most useful for our purpose; see \cite[Section 6.3]{fontana_factoring} and \cite[Proposition 4.3]{starloc}.
\item By Proposition \ref{prop:stronglyindep} (see also \cite[Theorem 6.3.1(4)]{fontana_factoring}) the two conditions ``$\Theta$ is independent'' and ``$\Theta$ is locally finite'' can be unified into the single one ``$\Theta$ is strongly independent''.
\item If $P$ is a nonzero prime of $D$, then there is exactly one $T\in\Theta$ such that $PT\neq T$: indeed, such a $T$ must exists since $\Theta$ is independent, while there cannot be two of them due to independence condition. In particular, $\Theta$ induces a partition on $\Max(D)$, called a \emph{Matlis partition} \cite[Section 6.3]{fontana_factoring}.
\item If $\Theta\in K$ (i.e., $\Theta=\{K\}$) then since $\Theta$ must be complete we must have also $D=K$, i.e., $D$ must be a field.
\end{enumerate}
\end{oss}

\begin{prop}\label{prop:caratt-jaffov2}
Let $T$ be a flat overring of $D$. Then, the following are equivalent:
\begin{enumerate}[(i)]
\item\label{prop:caratt-jaffov2:jaffov} $T$ is a Jaffard overring of $D$;
\item\label{prop:caratt-jaffov2:cdot} $T\cdot T^\ortog=K$;
\item\label{prop:caratt-jaffov2:family} $\{T,T^\ortog\}$ is a Jaffard family of $D$;
\item\label{prop:caratt-jaffov2:sigma} $\sigma(T)\cap\sigma(T^\ortog)=\{(0)\}$;
\item\label{prop:caratt-jaffov2:Psigma} for every nonzero $P\in\sigma(T)$, we have $PT^\ortog=T^\ortog$;
\item\label{prop:caratt-jaffov2:existsubloc} there is a sublocalization $A$ of $D$ such that $\{T,A\}$ is complete and $TA=K$.
\end{enumerate}
\end{prop}
\begin{proof}
\ref{prop:caratt-jaffov2:family} $\Longrightarrow$ \ref{prop:caratt-jaffov2:jaffov} follows from the definitions.

\ref{prop:caratt-jaffov2:jaffov} $\Longrightarrow$ \ref{prop:caratt-jaffov2:cdot} Let $\Theta$ be a Jaffard family containing $T$, and let $\Theta^\ortog(T):=\bigcap\{S\in\Theta\mid S\neq T\}$. For every nonzero prime ideal $P$ of $D$ out of $\Sigma(T)$, there is a unique $S\in\Theta$ such that $PS\neq S$; since $S$ is flat, we have $S\subseteq D_P$, and so
\begin{equation*}
\Theta^\ortog(T)\subseteq\bigcap\{P\in\Spec(D)\mid P\notin\Sigma(T)\}=T^\ortog.
\end{equation*}
Since $\Theta$ is strongly independent, $T\cdot\Theta^\ortog(T)=K$, and so $TT^\ortog=K$.

\ref{prop:caratt-jaffov2:cdot} $\Longrightarrow$ \ref{prop:caratt-jaffov2:family} Let $\Theta:=\{T,T^\ortog\}$. Clearly, $\Theta$ is locally finite and complete, while it is independent by Lemma \ref{lemma:sigmacap} (since $T$ is flat by hypothesis). Since $T\cap T^\ortog=D$, by \cite[Theorem 6.2.2(1)]{fontana_factoring} $T^\ortog$ is also flat; hence, $\Theta$ is a Jaffard family.

\ref{prop:caratt-jaffov2:cdot} $\iff$ \ref{prop:caratt-jaffov2:sigma} is exactly Lemma \ref{lemma:sigmacap}.

\ref{prop:caratt-jaffov2:cdot} $\Longrightarrow$ \ref{prop:caratt-jaffov2:existsubloc} is obvious. To show \ref{prop:caratt-jaffov2:existsubloc} $\Longrightarrow$ \ref{prop:caratt-jaffov2:jaffov}, it is enough to show that $\Theta:=\{T,A\}$ is a Jaffard family. By hypothesis, $\Theta$ is complete and locally finite, while it is independent by Lemma \ref{lemma:sigmacap}. In particular, by \cite[Theorem 6.2.2(1)]{fontana_factoring} $A$ is also flat, and thus $\Theta$ is a Jaffard family.

\ref{prop:caratt-jaffov2:family} $\Longrightarrow$ \ref{prop:caratt-jaffov2:Psigma} If $P\in\sigma(T)$, $P\neq(0)$, then $PT\neq T$; but since $\{T,T^\ortog\}$ is a Jaffard family, no nonzero prime can survive in both $T$ and $T^\ortog$, and so $PT^\ortog=T^\ortog$

\ref{prop:caratt-jaffov2:Psigma} $\Longrightarrow$ \ref{prop:caratt-jaffov2:cdot} Suppose that $T\cdot T^\ortog\neq K$: then, there is a nonzero $Q\in\Spec(D)$ such that $QTT^\ortog\neq TT^\ortog$, and so both $QT\neq T$ and $QT^\ortog\neq T^\ortog$. However, the first condition implies that $Q\in\sigma(T)$, contradicting the hypothesis.
\end{proof}

\begin{cor}\label{cor:caratt-jaff}
Let $\Theta$ be a complete and independent family of flat overrings of $D$. Then, $\Theta$ is a Jaffard family if and only if each $T\in\Theta$ is a Jaffard overring.
\end{cor}
\begin{proof}
If $\Theta$ is a Jaffard family then every $T\in\Theta$ is a Jaffard overring by definition. Conversely, suppose each $T\in\Theta$ is a Jaffard overring; by Proposition \ref{prop:stronglyindep} we only need to show that $\Theta$ is strongly independent. 

Fix $T\in\Theta$. If $S\in\Theta\setminus\{T\}$, then $\sigma(S)\cap\sigma(T)=\Sigma(S)\cap\Sigma(T)=\emptyset$, and thus $T^\ortog\subseteq S$. Therefore,
\begin{equation*}
T\left(\bigcap_{S\in\Theta\setminus\{T\}}S\right)\supseteq TT^\ortog=K
\end{equation*}
using Proposition \ref{prop:caratt-jaffov2}. Hence, $\Theta$ is strongly independent and thus a Jaffard family.
\end{proof}

We conclude this section with a lemma that will be useful later.
\begin{lemma}\label{lemma:dicotomia-DPS}
Let $\Theta$ be a complete and independent family of flat overrings of $D$ and let $P\neq(0)$ be a prime ideal of $D$. For every $S\in\Theta$, either $PS\neq S$ or $D_PS=K$.
\end{lemma}
\begin{proof}
Suppose that $PS=S$. Since $\Theta$ is complete, there is a $S'\in\Theta$ such that $PS'\neq S'$; since $S'$ is flat, by Lemma \ref{lemma:Sigmaiota} $S'\subseteq D_P$. Hence, $SS'\subseteq SD_P$; however, since $\Theta$ is independent $SS'=K$. Hence $K=SD_P$.
\end{proof}

\section{Pre-Jaffard families}\label{sect:preJaffard}
The hypothesis that a family is locally finite is usually very strong. To expand our reach beyond Jaffard families, we define a new class of families by weakening this condition.

\begin{defin}\label{def:preJaffard}
Let $\Theta$ be a family of overrings of $D$. We say that $\Theta$ is a \emph{pre-Jaffard family} of $D$ if:
\begin{itemize}
\item either $K\notin\Theta$ or $\Theta=\{K\}$;
\item every element of $\Theta$ is flat over $D$;
\item $\Theta$ is independent;
\item $\Theta$ is complete;
\item $\Theta$ is compact in the Zariski topology.
\end{itemize}
\end{defin}

\begin{oss}
We do not know any example of a family of overring satisfying the first three conditions of Definition \ref{def:preJaffard} but that is not compact; it is possible that the compactness condition is actually redundant.
\end{oss}

\begin{prop}\label{prop:prejaff-T2}
A pre-Jaffard family is Hausdorff, with respect to the inverse topology.
\end{prop}
\begin{proof}
Without loss of generality we can suppose that $K\notin\Theta$. Fix two distinct overrings $T,S\in\Theta$. Let $\mathcal{C}:=\{\B(x)\mid x\in T\setminus S\}\cup\{\B(y)\mid y\in S\setminus T\}$; we claim that $\mathcal{C}$ is a cover of $\Theta$.

Since $ST=K$, we have $S\subsetneq T$ and $T\subsetneq S$, and thus $T\setminus S$ and $S\setminus T$ are both nonempty; it follows that $S,T$ belong to some member of $\mathcal{C}$.

Let $A\in\Theta\setminus\{S,T\}$: then, 
\begin{equation*}
A\cap(T\setminus S)=A\cap T\cap(K\setminus S),
\end{equation*}
and thus if $A\cap(T\setminus S)=\emptyset$ then $A\cap T\subseteq S$. However, $(A\cap T)S=AS\cap TS=K$; hence, $A\cap(T\setminus S)\neq\emptyset$ and so there is an $x\in A\cap(T\setminus S)$, i.e., $A\in \B(x)$. Thus, any such $A$ belong to some member of $\mathcal{C}$, and $\mathcal{C}$ is a cover of $\Theta$.

Since $\Theta$ is compact in the Zariski topology, we can find $x_1,\ldots,x_n,y_1,\ldots,y_m$ such that $\{\B(x_1),\ldots,\B(x_n),\B(y_1),\ldots,\B(y_m)\}$ is a finite subcover. Let
\begin{equation*}
\Omega_1:=\bigcap_{i=1}^n\B(x_i)^c\quad\text{and}\quad\Omega_2:=\bigcap_{j=1}^m\B(y_j)^c;
\end{equation*}
then, $\Omega_1$ and $\Omega_2$ are both open in the inverse topology, since they are a finite intersection of subbasic open sets. Moreover, $S\in\Omega_1$ since $x_i\notin S$ for every $i$, while $T\in\Omega_2$ since $y_j\notin T$ for every $j$. Finally,
\begin{equation*}
\Omega_1\cap\Omega_2=\bigcap_{i=1}^n\B(x_i)^c\cap\bigcap_{j=1}^m\B(y_j)^c=\left(\bigcup_{i=1}^n\B(x_i)\cup\bigcup_{j=1}^m\B(y_j)\right)^c\subseteq\Theta^c,
\end{equation*}
that is, $\Omega_1\cap\Omega_2$ does not intersect $\Theta$. Hence, $\Omega_1\cap\Theta$ and $\Omega_2\cap\Theta$ are disjoint neighborhood of $S$ and $T$ in $\Theta$, with respect to the inverse topology. Thus, $\Theta$ is Hausdorff in the inverse topology.
\end{proof}

\begin{prop}\label{prop:Jaff->preJaff}
A Jaffard family of $D$ is a pre-Jaffard family.
\end{prop}
\begin{proof}
By definition, any Jaffard family of $D$ is independent, complete and composed of flat overrings. Furthermore, any locally finite family of overrings is compact, and thus a Jaffard family is also pre-Jaffard.
\end{proof}

\begin{prop}\label{prop:jaffov-isolated}
Let $\Theta$ be a pre-Jaffard family of $D$, and let $T\in\Theta$. Then, $T$ is a Jaffard overring of $D$ if and only if $\Theta\setminus\{T\}$ is compact in the Zariski topology. Furthermore, if this happens, then $T$ is isolated in $\Theta$, with respect to the inverse topology.
\end{prop}
\begin{proof}
If $T$ is a Jaffard overring, by Proposition \ref{prop:caratt-jaffov2} we have $TT^\ortog=K$, and in particular no overring of $D$ different from $K$ contains both $T^\ortog$ and $T$. Then, $\Theta^\uparrow\cap\{T^\ortog\}^\uparrow=(\Theta\setminus\{T\})^\uparrow$: however, since $\Theta^\uparrow$ and $\{T^\ortog\}^\uparrow$ are closed in the inverse topology (the former since $\Theta$ is compact by hypothesis), then also $(\Theta\setminus\{T\})^\uparrow$ is inverse-closed. Therefore, $\Theta\setminus\{T\}$, which is the set of minimal elements of $(\Theta\setminus\{T\})^\uparrow$, is compact with respect to the Zariski topology. This also shows that $T$ is isolated in $\Theta$, with respect to the inverse topology.

Suppose $\Theta\setminus\{T\}$ is compact, and let $A:=\bigcap\{S\mid S\in\Theta,S\neq T\}$. Then, $A$ is a sublocalization of $D$; moreover, since $T$ is flat and $\Theta\setminus\{T\}$ is compact, by \cite[Corollary 5]{compact-intersections} we have
\begin{equation*}
TA=T\bigcap_{S\in\Theta}S=\bigcap_{S\in\Theta}TS=K.
\end{equation*}
Hence, $T$ is a Jaffard overring by Proposition \ref{prop:caratt-jaffov2}.
\end{proof}

\begin{oss}
Proposition \ref{prop:jaffov-isolated} cannot be improved to a full equivalence between being a Jaffard overring and being and isolated point of $\Theta^\inverse$. Consider the ring $D$ defined in \cite[Example 2]{overrings-prufer-II}. Then, $D$ is a two-dimensional domain such that:
\begin{itemize}
\item all its finitely generated ideals are principal (i.e., $D$ is a B\'ezout domain); in particular, $D_M$ is a valuation domain for all maximal ideals $M$;
\item all its maximal ideals, except for one (say $M_\infty$), have height $1$;
\item $M_\infty$ is the radical of a principal ideal;
\item the unique nonzero, nonmaximal prime ideal $P$ is contained in a unique maximal ideal ($M_\infty$), but also in the union of all maximal ideals distinct from $M_\infty$.
\end{itemize}
Let $\Theta:=\{D_M\mid M\in\Max(D)\}$. Then, every $T\in\Theta$ is flat and $\Theta$ is complete; furthermore, $\Theta$ is independent (if $D_MD_N\neq K$, then $M\cap N$ should contain a nonzero prime ideal, a contradiction) and compact in the Zariski topology (since the localization map is continuous by \cite[Lemma 2.4]{dobbs_fedder_fontana} and so $\Theta$ is homeomorphic to $\Max(D)$); thus, $\Theta$ is a pre-Jaffard family.

Let $V:=D_{M_\infty}$: then, $T$ is not a Jaffard overring of $D$. Indeed, $V^\ortog=\bigcap\{D_M\mid M\in\Max(D),M\neq M_\infty\}$ is such that $PV^\ortog\neq V^\ortog$: otherwise, since $D$ is B\'ezout, there would be an $a\in P$ such that $aV^\ortog=V^\ortog$. However, there is also a maximal ideal $M\neq M_\infty$ such that $a\in M$: hence, $MV^\ortog=V^\ortog$, against the fact that $V^\ortog\subseteq D_M$ by construction. Hence, $PV^\ortog\neq V^\ortog$ and so $V^\ortog\subseteq D_P$, so that $VV^\ortog\subseteq VD_P=D_P$ (since $V=D_{M_\infty}\subseteq D_P$). By Proposition \ref{prop:caratt-jaffov2}, $V$ is not a Jaffard overring of $D$.

We claim that $V$ is isolated in $\Theta$, with respect to the inverse topology. Indeed, $M_\infty$ is the radical of a principal ideal, say $bD$; hence, $M_\infty$ is the unique $T\in\Theta$ such that $b^{-1}\notin T$, i.e., $\B(b^{-1})^c\cap\Theta=\{M_\infty\}$. However, $\B(b^{-1})^c$ is the complement of an open and compact subset of $\Over(D)$, and thus it is open in the inverse topology; hence, $V$ is isolated in $\Theta^\inverse$.
\end{oss}

The following two results show how to construct pre-Jaffard families from other such families by taking intersections.
\begin{lemma}\label{lemma:intersez-compact}
Let $\Theta$ be a family of overrings that is compact in the Zariski topology. Let $\Theta_1,\ldots,\Theta_n$ be subsets of $\Theta$, and for each $i$ let $S_i:=\bigcap\{T\mid T\in\Theta_i\}$. Then, the family
\begin{equation*}
\Theta':=\left(\Theta\setminus\bigcup_{i=1}^n\Theta_i\right)\cup\{S_1,\ldots,S_n\}
\end{equation*}
is compact in the Zariski topology.
\end{lemma}
\begin{proof}
Let ${\boldsymbol\Omega}:=\{\Omega_\alpha\}_{\alpha\in A}$ be an open cover of $\Theta'$. If $S_i\in\Omega_\alpha$, then $\Theta_i\subseteq\Omega_\alpha$ (since this holds for every subbasic open set $\B(x)$); hence, ${\boldsymbol\Omega}$ is also an open cover of $\Theta$. Since $\Theta$ is compact, we can find a finite subcover $\{\Omega_{\alpha_1},\ldots,\Omega_{\alpha_k}\}$ of $\Theta$; furthermore, for every $i$ we can find a $\beta_i\in A$ such that $S_i\in\Omega_{\beta_i}$. Then,  $\{\Omega_{\alpha_1},\ldots,\Omega_{\alpha_k},\Omega_{\beta_1},\ldots,\Omega_{\beta_n}\}$ is a finite subcover of $\Theta'$. Thus, $\Theta'$ is compact.
\end{proof}

\begin{prop}\label{prop:finite-intersez}
Let $\Theta$ be a pre-Jaffard family, and let $\Theta_1,\ldots,\Theta_n$ be pairwise disjoint subsets of $\Theta$ that are compact in the Zariski topology. For every $i$, let $S_i:=\bigcap\{T\mid T\in\Theta_i\}$. Then, the family
\begin{equation*}
\Theta':=\left(\Theta\setminus\bigcup_{i=1}^n\Theta_i\right)\cup\{S_1,\ldots,S_n\}
\end{equation*}
is a pre-Jaffard family.
\end{prop}
\begin{proof}
By induction, it is enough to prove the claim for $n=1$; let $S:=S_1$.

By construction, $\Theta'$ is complete, and by Lemma \ref{lemma:intersez-compact} it is compact in the Zariski topology. It is independent: indeed, take $T_1,T_2\in\Theta'$. If $T_1\neq S\neq T_2$ then $T_1T_2=K$ since $\Theta$ is independent, while if $S=T_2$ then
\begin{equation*}
T_1S=T_1\bigcap_{T\in\Theta_1}T=\bigcap_{T\in\Theta_1}T_1T=K
\end{equation*}
by \cite[Corollary 5]{compact-intersections}, since $\Theta_1$ is compact and every $T\in\Theta_1$ is in $\Theta$ and is different from $T_1$. Thus, $\Theta$ is independent.

We only need to prove that $S$ is flat. By construction, $S$ is a sublocalization. If $P\in\sigma(\Spec(S))$ is a nonzero prime, then $PT=T$ for every $T\in\Theta\setminus\{S\}$ (otherwise $K=TS\subseteq D_P$, a contradiction); on the other hand, there is a $S'\in\Theta_1$ such that $P\in\Sigma(S')$, and thus $P\in\Sigma(S)$. Hence, $\sigma(\Spec(S))=\Sigma(S)$, and $S$ is flat by Lemma \ref{lemma:Sigmaiota}.
\end{proof}

\section{Weak Jaffard families}\label{sect:weakJaffard}
Let $\Theta$ be a pre-Jaffard family. In general, we cannot expect the properties of a Jaffard family to hold also for $\Theta$; however, we want to show that at least some properties hold also under the weaker pre-Jaffard hypothesis. To do so, we want to proceed ``step-by-step'', isolating first the Jaffard overring belonging to $\Theta$; from a technical point of view, we need the following definition.
\begin{defin}
Let $\Theta$ be a family of overrings of $D$ and let $T_\infty\in\Theta$. We say that $\Theta$ is a \emph{weak Jaffard family of $D$ pointed at $T_\infty$} if:
\begin{itemize}
\item either $K\notin\Theta$ or $\Theta=\{K\}$;
\item $\Theta$ is complete and independent;
\item every $T\in\Theta\setminus\{T_\infty\}$ is a Jaffard overring of $D$;
\item $T_\infty$ is flat over $D$.
\end{itemize}
\end{defin}

\begin{lemma}\label{lemma:ThetaB}
Let $\Theta$ be a family of flat overrings of $D$, and let $B$ be an overring of $D$. Let $\Theta_B:=\{TB\mid T\in\Theta\}$
\begin{enumerate}[(a)]
\item\label{lemma:ThetaB:indep} If $\Theta$ is independent, $\Theta_B$ is independent.
\item\label{lemma:ThetaB:complete} If $\Theta$ is complete with respect to $D$, then $\Theta_B$ is complete with respect to $B$.
\item\label{lemma:ThetaB:flat} If every $T\in\Theta$ is flat as a $D$-module, every $TB\in\Theta_B$ is flat as a $B$-module.
\item\label{lemma:ThetaB:JaffOv} If $T$ is a Jaffard overring of $B$ and $T\neq K$, then $TB$ is a Jaffard overring of $B$.
\item\label{lemma:ThetaB:JaffFam} If $\Theta$ is a Jaffard family of $B$, then $\Theta_B\setminus\{K\}$ is a Jaffard family of $B$.
\end{enumerate}
\end{lemma}
\begin{proof}
\ref{lemma:ThetaB:indep} If $TB\neq T'B$ with $T\neq T$ in $\Theta$, then $(TB)(T'B)=(TT')B=K$.

\ref{lemma:ThetaB:complete} Let $I$ be a $B$-submodule of the quotient field $K$. Then, $IB=B$, and thus
\begin{equation*}
I=IB=\bigcap_{T\in\Theta}(IB)T=\bigcap_{T\in\Theta}I(BT)=\bigcap_{S\in\Theta_B}IS
\end{equation*}
so that $\Theta_B$ is complete with respect to $B$.


\ref{lemma:ThetaB:flat} Since $B$ is an overring, the extension $A\subseteq B$ is an epimorphism, and thus $TB\simeq T\otimes B$ \cite[Lemma 1.0]{lazard_flat}. The claim follows.

\ref{lemma:ThetaB:JaffOv} If $T$ is a Jaffard overring, then by Proposition \ref{prop:caratt-jaffov2} $\Theta:=\{T,T^\ortog\}$ is a Jaffard family of $D$. By the previous points, $\Theta_B=\{TB,T^\ortog B\}$ is complete, independent and formed by flat overrings; since it is clearly locally finite, it is a Jaffard family, and thus $TB$ is a Jaffard overring.

\ref{lemma:ThetaB:JaffFam} follows from the previous points and from Corollary \ref{cor:caratt-jaff}.
\end{proof}

A Jaffard family is always a weak Jaffard family, pointed to any of its elements. Analogously, a weak Jaffard family is pre-Jaffard, as we show next.
\begin{prop}\label{prop:weak->pre}
Let $\Theta$ be a weak Jaffard family of $D$ pointed at $S$.
\begin{enumerate}[(a)]
\item\label{prop:weak->pre:locfin} If $J$ is a proper ideal of $D$ and $JS=S$, then $JT\neq T$ for only finitely many $T\in\Theta$.
\item\label{prop:weak->pre:pre} $\Theta$ is a pre-Jaffard family of $D$.
\end{enumerate}
\end{prop}
\begin{proof}
\ref{prop:weak->pre:locfin} Let $B:=\bigcap\{D_M\mid M\in\V(J)\}$: we claim that $BS=K$. Indeed, $\V(J)$ is a compact subset of $\Spec(D)$, and thus
\begin{equation*}
BS=\left(\bigcap_{M\in\V(J)}D_M\right)S=\bigcap_{M\in\V(J)}D_MS=K
\end{equation*}
since if $J\subseteq M$ then $JT\neq T$ for some $T\in\Theta$ and thus $D_MS\supseteq TS=K$.

Consider the family $\Theta_B:=\{BT\mid T\in\Theta\}\setminus\{K\}$: by Lemma \ref{lemma:ThetaB}, $\Theta_B$ is a complete and independent set of flat overrings of $B$, and all its element except $BS$ are Jaffard overrings of $B$. By Corollary \ref{cor:caratt-jaff}, $\Theta_B$ is a Jaffard family of $B$, and thus it is locally finite. For every $T\in\Theta$ such that $JT\neq T$, also $JTB\neq TB$; hence, there are at most finitely many elements of $\Theta$ such that $JT\neq T$, as claimed.

\ref{prop:weak->pre:pre} We need only to show that $\Theta$ is compact, with respect to the Zariski topology. Let $\{\B(x_\alpha)\}_{\alpha\in\mathcal{I}}$ be an open cover of $\Theta$, and suppose $S\in \B(x)$. Let $J:=(D:_Dx)$: then, $JS=(S:_Sx)=S$ (using the flatness of $S$), and thus there are only finitely many $T\in\Theta$ such that $JT\neq T$; call them $T_1,\ldots,T_n$. Therefore, if $A\in\Theta\setminus\{T_1,\ldots,T_n\}$ then $JA=A$ and $(A:_Ax)=A$, i.e., $x\in A$; thus $\B(x)\setminus\Theta$ is finite. It follows that we can find a subcover $\{\B(x),\B(x_1),\ldots,\B(x_n)\}$ by choosing $x_1,\ldots,x_n$ such that $x_i\in T_i$. Since the cover was arbitrary, $\Theta\cup\{S\}$ is compact.
\end{proof}

Weak Jaffard families are much more ubiquitous than Jaffard families; the main reason is that a weak Jaffard family has a place to ``hide the singularities'' of $D$ (namely, the ring $T_\infty$ to which the family is pointed), while a Jaffard family does not have such a luxury. A first way in which weak Jaffard families arise is from a set of Jaffard overrings.
\begin{prop}\label{prop:jaffov->prejaff}
Let $\Theta$ be an independent set of Jaffard overrings of $D$, and let
\begin{equation*}
S:=\bigcap\{D_P\mid PT=T\text{~for every~}T\in\Theta\}.
\end{equation*}
Then, $\Theta\cup\{S\}$ is a weak Jaffard family of $D$ pointed at $S$.
\end{prop}
\begin{proof}
Let $P\in\Spec(D)$. If $PT\neq T$ for some $T\in\Theta$, then $T\subseteq D_P$ since $T$ is flat \cite[Theorem 1]{richamn_generalized-qr}; if $PT=T$ for every $T\in\Theta$, then $S\subseteq D_P$. Therefore, every localization $D_P$ of $D$ contains at least one element of $\Theta\cup\{S\}$. Hence, for every ideal $I$ of $D$,
\begin{equation*}
I=\bigcap_{M\in\Max(D)}ID_M\supseteq\bigcap_{T\in\Theta\cup\{S\}}IT\supseteq I.
\end{equation*}
Thus $\Theta\cup\{S\}$ is complete.

\medskip

If $T,T'\in\Theta$, $T\neq T'$, then $TT'=K$ by hypothesis. To show that $\Theta$ is independent, let $\Sigma:=\{P\in\Spec(T)\mid PT=T$ for every $T\in\Theta\}$. We claim that $\Sigma=\bigcap\{\Spec(D)\setminus\Sigma(T)\mid T\in\Theta\}$. Indeed, if $P\in\Sigma$ then $PT=T$ for every $T\in\Theta$, and thus $P$ is in the intersection; conversely, if $P\notin\Sigma(T)$ for every $T\in\Theta$, then (since each $T$ is flat) we have by Lemma \ref{lemma:Sigmaiota} $P\notin\sigma(T)$, and thus $PT=T$, so that $P\in\Sigma$. By Proposition \ref{prop:caratt-jaffov2}, $\Spec(D)\setminus\Sigma(T)=\Sigma(T^\ortog)\setminus\{(0)\}$; hence,
\begin{equation*}
\Sigma\cup\{(0)\}=\bigcap_{T\in\Theta}(\Sigma(T^\ortog)\setminus\{(0)\})\cup\{(0)\}=\bigcap_{T\in\Theta}\Sigma(T^\ortog)
\end{equation*}
Again by Proposition \ref{prop:caratt-jaffov2}, $T^\ortog$ is a Jaffard overring of $D$, and thus it is flat; hence, $\Sigma(T^\ortog)=\sigma(T^\ortog)$ is closed in the constructible topology, and thus also $\Sigma\cup\{(0)\}$ is closed in the constructible topology, and in particular it is compact. Fix now a $T\in\Theta$. Using the flatness of $T$, we have 
\begin{equation*}
TS=T\left(\bigcap_{P\in\Sigma\cup\{(0)\}}D_P\right)=K\cap\bigcap_{P\in\Sigma}TD_P=K,
\end{equation*}
since $D_P\supseteq T^\ortog$ for every nonzero $P\in\Sigma\supseteq\Sigma(T^\ortog)$. Hence, $\Theta\cup\{S\}$ is independent.

\medskip

Since every $T\in\Theta$ is a Jaffard overring, we only need to show that $S$ is flat. Suppose that $P\in\sigma(S)$: since $\Theta\cup\{S\}$ is complete and independent, we must have $PT=T$ for every $T\in\Theta$, and thus $P\in\Sigma$, so that, by definition $S\subseteq D_P$. Thus $P\in\Sigma(S)$ and $\sigma(S)=\Sigma(S)$; by Lemma \ref{lemma:Sigmaiota}, $S$ is flat.
\end{proof}

The previous two propositions provide a way to pass from a pre-Jaffard family to a weak Jaffard family.
\begin{defin}
If $\Theta$ is a pre-Jaffard family of $D$, we denote by $\Theta_J$ the set of Jaffard overrings contained in $\Theta$.
\end{defin}

\begin{prop}\label{prop:preJ->weakJ}
Let $\Theta$ be a pre-Jaffard family of $D$, and suppose that $\Theta_J\neq\Theta$. Let $S:=\bigcap\{T\mid T\in\Theta\setminus\Theta_J\}$. Then:
\begin{enumerate}[(a)]
\item\label{prop:preJ->weakJ:wJ} $\Theta_J\cup\{S\}$ is a weak Jaffard family of $D$ pointed at $S$;
\item\label{prop:preJ->weakJ:pJS} $\Theta\setminus\Theta_J$ is a pre-Jaffard family of $S$.
\end{enumerate}
\end{prop}
\begin{proof}
\ref{prop:preJ->weakJ:wJ} The set $\Theta_J$ is a set of Jaffard overrings of $D$; we claim that $S=\bigcap\{D_P\mid PT=T\text{~for every~}T\in\Theta\}$. Indeed, since every $T\in\Theta$ is flat we have
\begin{equation*}
S=\bigcap_{T\in\Theta\setminus\Theta_J}\bigcap_{P\in\Sigma(T)}D_P=\bigcap_{P\in\Sigma}D_P
\end{equation*}
where $\Sigma:=\bigcup\{\Sigma(T)\mid T\in\Theta\setminus\Theta_J\}$. Hence, $\Theta_J\cup\{S\}$ is a weak Jaffard family by Proposition \ref{prop:jaffov->prejaff}.

\ref{prop:preJ->weakJ:pJS} Let $\Theta':=\Theta\setminus\Theta_J$. Then, $\Theta'$ is an independent family of flat overrings of $S$, and is complete with respect to $S$, since if $I$ is an ideal of $S$ then
\begin{equation*}
I=IS=\bigcap_{T\in\Theta}IST=\bigcap_{T\in\Theta_J}IST\cap\bigcap_{T\in\Theta'}IST=\bigcap_{T\in\Theta'}IS
\end{equation*}
as $S\subseteq T$ if $T\in\Theta'$ while $ST=K$ if $T\in\Theta_J$. We thus need to show that $\Theta'$ is compact.

For every $T\in\Theta_J$, by Proposition \ref{prop:jaffov-isolated} $\Theta\setminus\{T\}$ is compact in the Zariski topology, and thus $\Lambda_T:=(\Theta\setminus\{T\})^\uparrow$ is closed in the inverse topology. Thus, also the intersection $\Lambda:=\bigcap\{\Lambda_T\mid T\in\Theta_J\}$ is closed in the inverse topology. However, since $\{S\}^\uparrow\cap\{S'\}^\uparrow=\{K\}$ for every $S\neq S'$ in $\Theta$, the set of minimal elements of $\Lambda$ is exactly $\Theta\setminus\Theta_J=\Theta'$; hence, $\Theta'$ is compact, as claimed.
\end{proof}

\section{The derived sequence}\label{sect:derived}
Let $\Theta$ be a pre-Jaffard family of $D$, and let $\Theta_J$ be the set of Jaffard overrings inside $\Theta$. If $\Theta=\Theta_J$, then by Corollary \ref{cor:caratt-jaff} $\Theta$ is a Jaffard family; on the other hand, if $\Theta\neq\Theta_J$ then by Proposition \ref{prop:preJ->weakJ} we can define an overring $T_1$ of $D$ such that:
\begin{itemize}
\item $\Theta_J\cup\{T_1\}$ is a weak Jaffard family of $D$;
\item $\Theta_1:=\Theta\setminus\Theta_J$ is a pre-Jaffard family of $T_1$.
\end{itemize}
In particular, we can repeat the same construction on $\Theta_1$: either $\Theta_1$ is a Jaffard family of $T_1$ or we can find an overring $T_2$ of $T_1$ (and so of $D$) such that $(\Theta_1)_J\cup\{T_2\}$ is a weak Jaffard family of $T_1$ and $\Theta_2:=\Theta_1\setminus(\Theta_1)_J$ is a pre-Jaffard family of $T_2$; then we can use the same construction on $T_2$, and so on. We now want to define rings $T_\alpha$ and subfamilies $\Theta_\alpha$ for every ordinal $\alpha$. 

To start, define $T_0:=D$ and $\Theta_0:=\Theta$.

Suppose that for every ordinal $\beta<\alpha$ we have defined a ring $T_\beta$ and a subset $\Theta_\beta\subseteq\Theta$ that is a pre-Jaffard family of $T_\beta$. Then:
\begin{itemize}
\item if $\alpha=\gamma+1$ is a successor ordinal, we define
\begin{equation*}
\Theta_\alpha:=\Theta_\gamma\setminus(\Theta_\gamma)_J;
\end{equation*}
\item if $\alpha$ is a limit ordinal, we define
\begin{equation*}
\Theta_\alpha:=\bigcap_{\beta<\alpha}\Theta_\beta.
\end{equation*}
\end{itemize}
In both cases, we define
\begin{equation*}
T_\alpha:=\bigcap\{S\mid S\in \Theta_\alpha\}
\end{equation*}
with the convention that $T_\alpha:=K$ if $\Theta_\alpha=\emptyset$.

\begin{prop}\label{prop:derived}
Preserve the notation above. Then:
\begin{enumerate}[(a)]
\item $\Theta_\alpha$ is a pre-Jaffard family of $T_\alpha$;
\item if $\alpha=\gamma+1$ is a successor ordinal, then $\Theta_\alpha\cup\{T_\alpha\}$ is a weak Jaffard family of $T_\gamma$ pointed at $T_\alpha$.
\end{enumerate}
\end{prop}
\begin{proof}
We proceed by transfinite induction: the claim is true by hypothesis for $\alpha=0$. Suppose that it holds for every $\beta<\alpha$. If $\alpha=\gamma+1$ is a successor ordinal, then the two statements are exactly Proposition \ref{prop:preJ->weakJ}. Suppose thus that $\alpha$ is a limit ordinal.

Each $T\in\Theta_\alpha$ is flat over $D$ and thus over $T_\alpha$; furthermore, $\Theta_\alpha$ is independent since it is contained in the independent set $\Theta$. Since also every $\Theta_\beta$ (for $\beta<\alpha$) is independent, as in the proof of Proposition \ref{prop:preJ->weakJ} we have 
\begin{equation*}
\Theta_\alpha^\uparrow=\left(\bigcap_{\beta<\alpha}\Theta_\beta\right)^\uparrow= \bigcap_{\beta<\alpha}\Theta_\beta^\uparrow
\end{equation*}
which is closed in the inverse topology since each $\Theta_\beta$ is compact with respect to the Zariski topology (being a pre-Jaffard family by inductive hypothesis), and so also $\Theta_\alpha$, which is the set of minimal elements of $\Theta_\alpha^\uparrow$, is compact in the Zariski topology. Thus, we only need to show that $\Theta_\alpha$ is complete. Let $P$ be a nonzero prime ideal of $T_\alpha$, and suppose that $PS=S$ for some $S\in\Theta_\alpha$: then, by Lemma \ref{lemma:dicotomia-DPS}, $SD_P=K$. Therefore, if $PS=S$ for all $S\in\Theta_\alpha$ then, by the flatness of $D_P$ and the compactness of $\Theta_\alpha$, by \cite[Corollary 5]{compact-intersections} we have
\begin{equation*}
D_P=D_PT_\alpha=D_P\bigcap_{S\in\Theta_\alpha}S=\bigcap_{S\in\Theta_\alpha}D_PS=K,
\end{equation*}
a contradiction. Hence $\Theta_\alpha$ is complete and thus it is a pre-Jaffard family.
\end{proof}

\begin{defin}
We call the family $\{T_\alpha\}$ defined in this way the \emph{derived sequence} with respect to $\Theta$.
\end{defin}

By construction, the derived sequence of $\Theta$ is an ascending chain of rings:
\begin{equation*}
D=T_0\subseteq T_1\subseteq T_2\subseteq\cdots\subseteq T_\omega\subseteq\cdots,
\end{equation*}
which corresponds to a descending chain of sets of overrings:
\begin{equation*}
\Theta=\Theta_0\supseteq\Theta_1\supseteq\Theta_2\supseteq\cdots\supseteq\Theta_\omega\supseteq\cdots.
\end{equation*}

\begin{prop}
Preserve the notation above. There is an ordinal $\alpha$ such that $T_\alpha=T_{\alpha'}$ for all $\alpha'>\alpha$ (equivalently, such that $\Theta_\alpha=\Theta_{\alpha'}$ for all $\alpha'>\alpha$).
\end{prop}
\begin{proof}
Note that, if $T_\alpha=T_{\alpha+1}$, then $T_\alpha=T_{\alpha'}$ for all $\alpha'>\alpha$; thus, suppose by contradiction that $T_\alpha\supsetneq T_{\alpha+1}$ for all $\alpha$. Then, $T_{\alpha+1}\setminus T_\alpha$ is nonempty for all $\alpha$; but since all the $T_\alpha$ are contained inside $K$, this is impossible if the cardinality of $\alpha$ is larger than the cardinality of $K$.
\end{proof}

\begin{defin}
We call the minimal ordinal $\alpha$ such that $T_\alpha=T_{\alpha'}$ the \emph{Jaffard degree} of the family $\Theta$, and we call $T_\alpha$ the \emph{dull limit} of $\Theta$. We say that $\Theta$ is:
\begin{itemize}
\item a \emph{sharp family} if $T_\alpha=K$;
\item a \emph{dull family} if $T_\alpha\neq K$.
\end{itemize}
\end{defin}

Equivalently, $\Theta$ is a sharp family if $\Theta_\alpha=\emptyset$ for some $\alpha$, while it is a dull family otherwise.

The terminology sharp/dull family is chosen in analogy with \cite{loper-lucas-factoring-AD} and \cite{HK-Olb-Re}, where sharp and dull domains (and, in correspondence, sharp and dull degrees) are defined, respectively, for almost Dedekind domains and for one-dimensional Pr\"ufer domains; our definition can be seen as a wide generalization of their concept. However, we do not use the terminology ``sharp degree'' and ``dull degree'', both because the definition of Jaffard degree unifies them and because there is actually a small difference in the sharp case. See Section \ref{sect:dim1} for a more detailed discussion.

\begin{ex}
Let $\Theta$ be a Jaffard family of $D$, with $D\neq K$. Then, $\Theta=\Theta_J$, and so $\Theta_1=\emptyset$; thus, $T_1=K=T_\alpha$ for all ordinals $\alpha>0$. Hence, $\Theta$ is sharp with Jaffard degree $1$.
\end{ex}

\begin{ex}
Let $\Theta$ be a weak Jaffard family of $D$ pointed at $S$. Then, $\Theta_1=\{S\}$, and so $T_1=S$; on the other hand, $\Theta_2=\emptyset$ and so $T_2=K$. Thus $\Theta$ is sharp with Jaffard degree $2$.
\end{ex}

\begin{ex}\label{ex:almded}
Let $D$ be an almost Dedekind domain with only finitely many maximal ideals that are not finitely generated, say $M_1,\ldots,M_n$. (Those rings do indeed exists: see \cite{loper_sequence}.) Let $\Theta:=\{D_M\mid M\in\Max(D)\}$: then, $\Theta$ is a pre-Jaffard family of $D$ (see Proposition \ref{prop:Thetadim1} below).

If $P$ is a maximal ideal of $D$ different from the $M_i$, then $D_P$ is a Jaffard overring of $D$, since $\Max(D)\setminus\{P\}$ is compact and thus $D_PD_P^\ortog=K$. On the other hand, each $D_{M_i}$ is not a Jaffard overring, since $D_Q^\ortog=\bigcap\{D_Q\mid Q\in\Max(D)\setminus\{M_i\}\}=D$; in particular, there is no weak Jaffard family of $D$ that can contain at the same time every $D_{M_i}$. Furthermore, the family $\Theta':=\{D_M\mid M\in\Max(D),M\neq M_1,\ldots,M_n\}$ is strongly independent (since every $D_M$ is a Jaffard overring and $D_M^\ortog$ contains the intersection of all $T\in\Theta'\setminus\{D_M\}$), but it is not locally finite, since otherwise the whole $\Theta=\Theta'\cup\{D_{M_1},\ldots,D_{M_n}\}$ would be locally finite and thus a Jaffard family.

The set $\Theta_1$ is equal to $\{M_1,\ldots,M_n\}$ and thus it is finite; moreover, $T_1=D_{M_1}\cap\cdots\cap D_{M_n}$ is a semilocal almost Dedekind domain, and thus a PID. Therefore, $\Theta_2=\emptyset$ and $T_2=K$, so that $\Theta$ is sharp with Jaffard degree $2$.
\end{ex}

\begin{ex}
Let $D$ be the ring of all algebraic integers, i.e., the integral closure of $\insZ$ in $\overline{\insQ}$. Then, $D$ is a one-dimensional B\'ezout (in particular, Pr\"ufer) domain such that none of its maximal ideals are finitely generated, nor any nonzero primary ideal is finitely generated.

Therefore, $\Theta:=\{D_M\mid M\in\Max(D)\}$ is a pre-Jaffard family of $D$ (as in the previous example), but none of its elements are Jaffard overrings: hence $\Theta_1=\emptyset$ and $T_1=D=T_0$. Therefore, $\Theta$ is dull with Jaffard degree $0$ and its dull limit is $D$ itself.

Let now $\Theta':=\{D_M\mid M\in\Max(D),2\in M\}\cup\{D[1/2]\}$. Then, $\Theta'$ is obtained from $\Theta$ with the construction of Proposition \ref{prop:finite-intersez} applied to $\Theta_1:=\{S\in\Theta\mid 1/2\in S\}=\B(1/2)\cap\Theta$ (which is compact), and thus is a pre-Jaffard family. The ring $D[1/2]$ is a Jaffard overring of $D$, since it belongs to the Jaffard family $\{D[1/p]\mid p$ is a prime number$\}$, while no other element of $\Theta'$ is a Jaffard overring; hence, $(\Theta')_1=\{D_M\mid M\in\Max(D),2\in M\}$, while $(\Theta')_2=(\Theta')_1$. Hence $\Theta'$ is dull with Jaffard degree $1$, and its dull limit is 
\begin{equation*}
T_1=\bigcap_{\substack{M\in\Max(D)\\ 2\in M}}D_M=D[1/3,1/5,\ldots]=D[1/p\mid p\neq 2\text{~is a prime number}].
\end{equation*}
\end{ex}

\begin{oss}
Note that, if $D$ is not a field and $\Theta$ is sharp, then the Jaffard degree of $D$ cannot be $0$.
\end{oss}

\section{Stable operations}\label{sect:stable}
Let $T$ be an overring of $D$. Then, $\inssubmod_D(K)\subseteq\inssubmod_T(K)$, and the image of any $T$-submodule of $K$ by any semistar operation on $D$ is still a $T$-module. Then, we have two ways to relate the semistar operations on $D$ and $T$: the first one is the restriction map
\begin{equation*}
\begin{aligned}
\psi_T\colon\inssemistar(D) & \longrightarrow\inssemistar(T),\\
\star & \longmapsto \star|_{\inssubmod_T(K)},
\end{aligned}
\end{equation*}
while the second is the extension map
\begin{equation*}
\begin{aligned}
\phi_T\colon\inssemistar(T) & \longrightarrow\inssemistar(D),\\
\star & \longmapsto \phi_T(\star)\colon\begin{aligned}
\inssubmod_D(K) & \longrightarrow\inssubmod_D(K),\\
I& \longmapsto (IT)^\star.
\end{aligned}
\end{aligned}
\end{equation*}

If now we have a family $\Theta$ of overrings of $D$, then we can put together the maps relative to each member of the family: we obtain a restriction map
\begin{equation*}
\begin{aligned}
\Psi_\Theta\colon\inssemistar(D) & \longrightarrow\prod_{T\in\Theta}\inssemistar(T),\\
\star & \longmapsto (\Psi_T(\star))
\end{aligned}
\end{equation*}
and an extension map
\begin{equation*}
\begin{aligned}
\Phi_\Theta\colon\prod_{T\in\Theta}\insstable(T) & \longrightarrow\insstable(D),\\
(\star^{(T)})_{T\in\Theta} & \longmapsto \inf_{T\in\Theta}\Phi_T(\star^{(T)}).
\end{aligned}
\end{equation*}
All these maps are order-preserving when $\inssemistar(D)$ and $\inssemistar(T)$ are endowed with the natural order, and when the product is endowed with the product order.

\begin{prop}\label{prop:PsiPhi}
Let $\Theta$ be a complete and independent family of overrings of $D$. Then, $\Psi_\Theta\circ\Phi_\Theta$ is the identity on $\prod_{T\in\Theta}\inssemistar(T)$.
\end{prop}
\begin{proof}
For every $T$, let $\star^{(T)}\in\inssemistar(T)$, and let $\star:=\Phi_\Theta((\star^{(T)})_{T\in\Theta})$. Fix $S\in\Theta$ and let $I\in\inssubmod_S(D)$. Then, $I=IS\neq(0)$; since $I$ is complete, we have
\begin{equation*}
I^{\star_S}=I^\star=\bigcap_{T\in\Theta}(IT)^{\star^{(T)}}=(IS)^{\star^{(S)}}\cap\bigcap_{T\in\Theta\setminus\{S\}}(IST)^{\star^{(T)}}=I^{\star^{(S)}}
\end{equation*}
as $ST=K$ for every $T\in\Theta\setminus\{S\}$ (since $\Theta$ is independent). The claim is proved.
\end{proof}

\begin{defin}\label{defin:stable-pres}
Let $\Theta$ be a family of overrings of $D$. We say that $\Theta$ is \emph{stable-preserving} if, for every $\star\in\insstable(D)$ and every $I\in\inssubmod_D(K)$, we have
\begin{equation*}
I^\star=\bigcap_{T\in\Theta}(IT)^\star.
\end{equation*}
\end{defin}

A stable semistar operation is uniquely determined by its action on proper ideals of $D$. Hence, if $\star$ is a stable semistar operation fixing $D$, then the notion of extension of a star operation studied in \cite{starloc} can be used to show that if $\Theta$ is a Jaffard family then $I^\star=\bigcap_{T\in\Theta}(IT)^\star$ (see, in particular, \cite[Theorems 5.4 and 5.6]{starloc}); a similar result, without the hypothesis $D=D^\star$, can be shown joining the results in Sections 3 and 6 of \cite{length-funct} (passing through length functions), so that any Jaffard family is stable-preserving. We want to generalize this case, but we first point out why stable-preserving properties are useful.

\begin{prop}\label{prop:stablepres}
Let $\Theta$ be a stable-preserving family of flat overrings of $D$. Then, $\Psi_\Theta$ and $\Phi_\Theta$ establish an order-preserving isomorphism between $\insstable(D)$ and $\prod\{\insstable(T)\mid T\in\Theta\}$.
\end{prop}
\begin{proof}
If $\star$ is a stable semistar operation, then the restriction $\Psi_T(\star)$ is stable for every overring $T$; conversely, the infimum of a family of restriction of stable operations is still stable, since
\begin{align*}
(I\cap J)^{\Phi_\Theta(\star^{(T)})}= & \bigcap_{T\in\Theta}((I\cap J)T)^{\star^{(T)}}=\bigcap_{T\in\Theta}((IT\cap JT))^{\star^{(T)}}=\\
= & \bigcap_{T\in\Theta}(IT)^{\star^{(T)}}\cap(JT)^{\star^{(T)}}=I^{\Phi_\Theta(\star^{(T)})}\cap J^{\Phi_\Theta(\star^{(T)})},
\end{align*}
using the flatness of the members of $\Theta$. Hence, the maps $\Phi_\Theta$ and $\Psi_\Theta$ restrict to maps from $\insstable(D)$ to $\prod_{T\in\Theta}\insstable(T)$.

By Proposition \ref{prop:PsiPhi}, $\Psi_\Theta\circ\Phi_\Theta$ is the identity. Let now $\star\in\insstable(D)$. Then,
\begin{equation*}
I^{\Phi_\Theta\circ\Psi_\Theta(\star)}=\bigcap_{T\in\Theta}(IT)^\star=I^\star
\end{equation*}
since $\Theta$ is stable-preserving. Hence, $\Phi_\Theta\circ\Psi_\Theta$ is the identity of $\insstable(D)$, and thus $\Phi_\Theta$ and $\Psi_\Theta$ are isomorphism.
\end{proof}

\begin{prop}\label{prop:weakJaff-stable}
A weak Jaffard family is stable-preserving.
\end{prop}
\begin{proof}
Let $\Theta$ be a weak Jaffard family pointed at $T_\infty$. Fix any $\star\in\insstable(D)$, and let $\sharp$ be the map $I\mapsto\bigcap_{T\in\Theta}(IT)^\star$. Then, $\sharp$ is stable, and $\star\leq\sharp$; in particular, if $1\in I^\star$ then $1\in I^\sharp$.

Conversely, let $I\subseteq D$ be such that $1\in I^\sharp$; without loss of generality, we can suppose that $I=I^\star$. Let $T\in\Theta\setminus\{T_\infty\}$: then, then, $T$ is a Jaffard overring of $D$, and thus $\{T,T^\ortog\}$ is a Jaffard family of $D$ by Proposition \ref{prop:caratt-jaffov2}. Hence,
\begin{equation*}
IT=I^\star T=(IT\cap IT^\ortog)^\star T=((IT)^\star\cap(IT^\ortog)^\star)T=(IT)^\star\cap(IT^\ortog)^\star T.
\end{equation*}
By definition, $I^\sharp\subseteq(IT)^\star$, and thus $1\in(IT)^\star$; on the other hand, $TT^\ortog=K$ and thus $(IT^\ortog)^\star T=K$. Thus, $1\in IT$ and $IT=T$.

Since $\Theta$ is complete, we thus have $I=(IT_\infty\cap D)$, and so
\begin{equation*}
IT_\infty=I^\star T_\infty=(IT_\infty\cap D)^\star T_\infty=(IT_\infty)^\star\cap D^\star T_\infty.
\end{equation*}
Again, by construction $1$ belongs to both $(IT_\infty)^\star$ and $D^\star T_\infty$, and thus $1\in IT_\infty$, so that $IT_\infty=T_\infty$. Hence, $IT=T$ for every $T\in\Theta$, and thus $I=D$. Therefore, for every $I\subseteq D$ we have $1\in I^\star$ if and only if $1\in I^\sharp$; since $\star$ and $\sharp$ are stable, it follows that $\star=\sharp$, as claimed. Thus, $\Theta$ is stable-preserving, as claimed.
\end{proof}

\begin{teor}\label{teor:stagJ-stable}
Let $\Theta$ be a Jaffard family, $\alpha$ an ordinal, and let $\Theta':=(\Theta\setminus\Theta_\alpha)\cup\{T_\alpha\}$. Then, $\Theta'$ is stable-preserving.
\end{teor}
\begin{proof}
For every $\beta\leq\alpha$, let $\Lambda_\beta:=(\Theta\setminus\Theta_\beta)\cup\{T_\beta\}$: we want to show by induction that $\Lambda_\beta$ is stable-preserving. Note that each $\Lambda_\beta$ is complete and, by definition, $\Lambda_\alpha=\Theta'$.

If $\beta=0$ then $\Lambda_0=(\Theta\setminus\Theta)\cup\{T_0\}=\{T_0\}=\{D\}$ is clearly stable-preserving; suppose thus $\beta>0$, and suppose that the claim holds for every $\gamma<\beta$; we distinguish two cases.

If $\beta=\gamma+1$ is a successor ordinal, then $\Theta_\beta=\Theta_\gamma\setminus(\Theta_\gamma)_J$ and thus
\begin{equation*}
\Lambda_\beta=(\Theta\setminus(\Theta_\gamma\setminus(\Theta_\gamma)_J))\cup\{T_\beta\}=(\Theta\setminus\Theta_\gamma)\cup(\Theta_\gamma)_J\cup\{T_\beta\}.
\end{equation*}
Let $\Lambda':=\Theta\setminus\Theta_\gamma$, and take a stable semistar operation on $D$. By inductive hypothesis, $\Lambda_\gamma=\Lambda'\cup\{T_\gamma\}$ is stable-preserving, and thus
\begin{equation*}
I^\star=\bigcap_{A\in\Lambda'}(IA)^\star\cap (IT_\gamma)^\star.
\end{equation*}
Moreover, by construction, $(\Theta_\gamma)_J\cup\{T_\beta\}$ is a weak Jaffard family of $T_\gamma$ pointed at $T_\beta$; by Proposition \ref{prop:weakJaff-stable}, it is stable-preserving on $T_\beta$, and thus
\begin{equation*}
(IT_\gamma)^\star=\bigcap_{A\in\Theta_\gamma\cup\{T_\beta\}}(IT_\gamma A)^\star=\bigcap_{A\in\Theta_\gamma\cup\{T_\beta\}}(IA)^\star,
\end{equation*}
so that
\begin{equation*}
I^\star=\bigcap_{A\in\Lambda'}(IA)^\star\cap\bigcap_{A\in\Theta_\gamma\cup\{T_\beta\}}(IA)^\star=\bigcap_{A\in\Lambda_\beta}(IA)^\star.
\end{equation*}
Hence, $\Lambda_\beta$ is stable-preserving.

Suppose now that $\beta$ is a limit ordinal: then, $\Theta_\beta=\bigcap_{\gamma<\beta}\Theta_\gamma$, and thus
\begin{equation*}
\Lambda_\beta=\left(\Theta\setminus\bigcap_{\gamma<\beta}\Theta_\gamma\right)\cup\{T_\beta\}=\bigcup_{\gamma<\beta}(\Theta\setminus\Theta_\gamma)\cup\{T_\beta\}.
\end{equation*}
Let $\star$ be a stable semistar operation, and let $\sharp$ be the map 
\begin{equation*}
\sharp:I\mapsto\bigcap_{A\in\Lambda_\beta}(IA)^\star=\left[\bigcap_{\gamma<\beta}\bigcap_{A\in\Theta_\gamma} (IA)^\star\right]\cap(IT_\beta)^\star.
\end{equation*}
Then, $\sharp$ is a stable semistar operation, and $I^\star\subseteq I^\sharp$ for all ideals $I$ (as $I^\star$ is contained in all $(IA)^\star$ and in $(IT_\beta)^\star$). We claim that it is equal to $\star$, and to do so it is enough to show that if $1\in I^\sharp$ then also $1\in I^\star$, where $I$ is a proper ideal of $D$ (this follows from condition (4) of \cite[Theorem 2.6]{anderson_two_2000}).

Take thus a proper ideal $I$ such that $1\in I^\sharp$, and let $\Gamma:=\{\gamma<\beta\mid IT\neq T\text{~for some~}T\in\Theta_\gamma\setminus\Theta_{\gamma+1}\}$.

Suppose that $\Gamma$ is nonempty: then, it has a minimum $\gamma$. Since $\Lambda_\gamma$ is complete and $IS=S$ for all $S\in\Theta_\delta$ with $\delta<\gamma$, we must have $I=IT_\gamma\cap D$; as above, it follows that $I^\star=(IT_\gamma)^\star\cap D^\star$ Let $T\in\Theta_\gamma\setminus\Theta_{\gamma+1}$: then, $T$ is a Jaffard overring of $T_\gamma$. Let
\begin{equation*}
A:=\bigcap\{(T_\gamma)_P\mid P\in\Spec(T_\gamma),PT=T\},
\end{equation*}
that is, $A=T^\ortog$ with respect to $T_\gamma$. Then, $TA=K$ and $J=JT\cap JA$ for all $T_\gamma$-submodules $J$ of $K$. Thus,
\begin{equation*}
(IT_\gamma)^\star T=(IT_\gamma T\cap IT_\gamma A)^\star T=(IT)^\star\cap((IA)^\star)T=(IT)^\star.
\end{equation*}
Therefore,
\begin{equation*}
I^\star T=((IT_\gamma)^\star\cap D^\star)T=(IT)^\star\cap D^\star T
\end{equation*}
contains $1$ since it contains $I^\sharp$. Since $T$ was arbitrary in $\Theta_\gamma\setminus\Theta_{\gamma+1}$, this is a contradiction.

Therefore, $\Gamma$ must be empty. Since $\beta$ is a limit ordinal, $\Lambda_\beta$ is also equal to $\bigcup_{\gamma<\beta}(\Theta_\gamma\setminus\Theta_{\gamma+1})$; therefore, since $\Lambda_\beta$ is complete and $I$ is proper, we must have $I=IT_\beta\cap D$; therefore, $I^\star=(IT_\beta\cap D)^\star=(IT_\beta)^\star\cap D^\star$ since $\star$ is stable. However, $1\in(IT_\beta)^\star$ since $(IT_\beta)^\star$ contains $I^\sharp$, while obviously $1\in D^\star$; hence, $1\in I^\star$.

By induction, it follows that $\Lambda_\alpha=\Theta'$ is stable-preserving, as claimed.
\end{proof}

\begin{cor}\label{cor:stable-derived}
Let $\Theta$ be a Jaffard family, $\alpha$ an ordinal, and let $\Theta':=(\Theta\setminus\Theta_\alpha)\cup\{T_\alpha\}$. Then:
\begin{enumerate}
\item for every stable semistar operation $\star$ on $D$, we have $I^\star=\bigcap\{(IT)^\star\mid T\in\Theta'\}$;
\item $\insstable(D)\simeq\prod\{\insstable(T)\mid T\in\Theta'\}$.
\end{enumerate}
\end{cor}
\begin{proof}
The first part follows by joining Theorem \ref{teor:stagJ-stable} with Definition \ref{defin:stable-pres}, the second part from Theorem \ref{teor:stagJ-stable} and Proposition \ref{prop:stablepres}.
\end{proof}

From the correspondence between stable semistar operations and length functions we have the following.
\begin{prop}
Let $D$ be an integral domain and let $\Theta$ be a stable-preserving family of flat overrings of $D$. Then, for every $\ell\in\inslengthsing(D)$, we have
\begin{equation*}
\ell=\sum_{T\in\Theta}\ell\otimes T.
\end{equation*}
\end{prop}

In particular, this holds for Jaffard families, as was proved in \cite[Theorem 3.10]{length-funct}; likewise, the analogue of Corollary \ref{cor:stable-derived} holds.
\begin{cor}
Let $\Theta$ be a Jaffard family, $\alpha$ an ordinal, and let $\Theta':=(\Theta\setminus\Theta_\alpha)\cup\{T_\alpha\}$. Then:
\begin{enumerate}
\item for every length function $\ell$ on $D$, we have $\ell=\sum\{\ell\otimes T\mid T\in\Theta'\}$;
\item $\inslengthsing(D)\simeq\prod\{\inslengthsing(T)\mid T\in\Theta'\}$.
\end{enumerate}
\end{cor}

Obviously, the previous results are at their strongest when $\alpha$ is the Jaffard degree of $\Theta$, so that $T_\alpha$ is the dull limit of $\Theta$.

\section{The dimension $1$ case}\label{sect:dim1}
In this section, we specialize the results of the previous sections to domains of dimension $1$. In this case, there is a natural pre-Jaffard family to consider.
\begin{prop}\label{prop:Thetadim1}
Let $D$ be a domain of dimension $1$, and let $\Theta:=\{D_M\mid M\in\Max(D)\}$. Then, $\Theta$ is a pre-Jaffard family of $D$.
\end{prop}
\begin{proof}
The family $\Theta$ is clearly complete, independent (no nonzero prime survives in $D_M$ and in $D_N$ for $M\neq N$) and composed of flat overrings. Furthermore, the localization map $\lambda:\Spec(D)\longrightarrow\Over(D)$ is a homeomorphism between $\Spec(D)$ and $\lambda(\Spec(D))$ when the spaces are endowed with the respective Zariski topologies \cite[Lemma 2.4]{dobbs_fedder_fontana}; in particular, $\Theta=\lambda(\Max(D))$ is compact. Hence, $\Theta$ is a pre-Jaffard family.
\end{proof}

\begin{defin}
Let $D$ be a one-dimensional integral domain, and let $\Theta:=\{D_M\mid M\in\Max(D)\}$. We say that $D$ is:
\begin{itemize}
\item \emph{ultimately sharp} if $\Theta$ is sharp;
\item \emph{ultimately dull} if $\Theta$ is dull.
\end{itemize}
\end{defin}

The second reason why the dimension $1$ hypothesis is powerful is that we can improve Proposition \ref{prop:jaffov-isolated}.
\begin{prop}\label{prop:jaffov-isolated-dim1}
Let $D$ be a domain of dimension $1$, and let $\Theta:=\{D_M\mid M\in\Max(D)\}$. Let $M\in\Max(D)$. Then, the following are equivalent:
\begin{enumerate}[(i)]
\item\label{prop:jaffov-isolated-dim1:jaff} $D_M$ is a Jaffard overring of $D$;
\item\label{prop:jaffov-isolated-dim1:thetacomp} $\Theta\setminus\{D_M\}$ is compact, with respect to the Zariski topology;
\item\label{prop:jaffov-isolated-dim1:maxcomp} $\Max(D)\setminus\{M\}$ is compact, with respect to the Zariski topology;
\item\label{prop:jaffov-isolated-dim1:isolated} $M$ is isolated in $\Max(D)$, with respect to the inverse topology.
\end{enumerate}
\end{prop}
\begin{proof}
The equivalence of \ref{prop:jaffov-isolated-dim1:jaff} and \ref{prop:jaffov-isolated-dim1:thetacomp} follows from Proposition \ref{prop:jaffov-isolated} (and Proposition \ref{prop:Thetadim1}), while the equivalence of \ref{prop:jaffov-isolated-dim1:thetacomp} and \ref{prop:jaffov-isolated-dim1:maxcomp} from the fact that $\Theta\setminus\{D_M\}\simeq\Max(D)\setminus\{M\}$ via the localization map. Again by Proposition \ref{prop:jaffov-isolated}, \ref{prop:jaffov-isolated-dim1:jaff} implies \ref{prop:jaffov-isolated-dim1:isolated}.

Suppose \ref{prop:jaffov-isolated-dim1:isolated} holds. Then, $\Max(D)\setminus\{M\}$ is closed in $\Max(D)$, with respect to the inverse topology. Since $D$ has dimension $1$, it follows that $\Spec(D)\setminus\{M\}$ is closed in the inverse topology, and thus that $\Max(D)\setminus\{M\}$ is compact. Hence, \ref{prop:jaffov-isolated-dim1:maxcomp} holds and all the conditions are equivalent.
\end{proof}

Recall that, for a topological space $X$, $\isolated(X)$ and $\limitp(X)$ are, respectively, the set of isolated points and the set of limit points of $X$. Proposition \ref{prop:jaffov-isolated-dim1} allows to describe the derived sequence in a purely topological way.

\begin{teor}\label{teor:derived-1dim}
Let $D$ be a one-dimensional domain, let $\Theta:=\{D_M\mid M\in\Max(D)\}$ and let $X:=\Max(D)^\inverse$. Let $\{T_\alpha\}$ be the derived sequence of $\Theta$ and let $\{\Theta_\alpha\}$ be the corresponding chain of subsets of $\Theta$. Then, for every ordinal $\alpha$ and every $M\in\Max(D)$, we have $MT_\alpha\neq T_\alpha$ if and only if $M\in\limitp^\alpha(X)$, and $\Theta_\alpha=\{D_M\mid M\in\limitp^\alpha(X)\}$.
\end{teor}
\begin{proof}
Let $\Lambda_\alpha:=\{M\in\Max(D)\mid MT_\alpha\neq T_\alpha\}$; then, $\{\Lambda_\alpha\}$ is a descending chain of subsets of $\Max(D)$, and we need to show that $\Lambda_\alpha=\limitp^\alpha(X)$. By definition, $\Theta_\alpha$ is a pre-Jaffard family of $T_\alpha$; it follows that $M\in\Lambda_\alpha$ if and only if $D_M\in\Theta_\alpha$.

We proceed by transfinite induction. For $\alpha=0$, $\Theta_0=\Max(D)$ and $\limitp^0(X)=X$, so the claim is proved. Suppose that the claim holds for every $\beta<\alpha$; we distinguish two cases.

Suppose first that $\alpha=\gamma+1$ is a successor ordinal. Then, by hypothesis, $\{M\in\Max(D)\mid MT_\gamma\neq T_\gamma\}=\limitp^\gamma(X)$; hence, the restriction map $\rho:\Max(T_\gamma)\longrightarrow\Max(D)$, $P\mapsto P\cap D$ establishes a homeomorphism between $\Max(T_\gamma)$ and its image $\limitp^\gamma(X)=\Lambda_\gamma(X)$ both in the Zariski and in the inverse topology. By definition and by Proposition \ref{prop:jaffov-isolated-dim1}, $\Theta_\alpha=\Theta_{\gamma+1}=\{(T_\gamma)_P\mid P\in(\Max(T_\gamma))^\inverse\}$, i.e., given a maximal ideal $M$ of $D$, we have $D_M\in\Theta_\alpha$ if and only if $MT_\gamma$ is a limit point of $(\Max(T_\gamma))^\inverse$. Since $\rho$ is a homeomorphism in the inverse topology, this is equivalent to saying that $M\in\limitp^{\gamma+1}(X)=\limitp^\alpha(X)$; that is, $D_M\in\Theta_\alpha$ if and only if $M\in\limitp^\alpha(X)$. By the remark at the beginning of the proof, we have our claim.

Suppose now that $\alpha$ is a limit ordinal. If $M\in\Lambda_\alpha$, then $M\in\Lambda_\beta$ for all $\beta<\alpha$, and thus by induction $M\in\limitp^\beta(X)$ for all $\beta<\alpha$; by definition, this is exactly the condition $M\in\limitp^\alpha(X)$. Conversely, if $M\in\limitp^\alpha(X)$ then $M\in\limitp^\beta(X)$ for all $\beta<\alpha$, and thus by induction $D_M\in\Theta_\beta$ for all $\beta<\alpha$; therefore, $D_M\in\Theta_\alpha$ by definition and $M\in\Lambda_\alpha$. Thus $M\in\Lambda_\alpha$ if and only if $M\in\limitp^\alpha(X)$, and the claim is proved.
\end{proof}

\begin{cor}\label{cor:CantBend}
Let $D$ be a one-dimensional integral domain. Then, the Jaffard degree of $\Theta$ is equal to the Cantor-Bendixson rank of $\Max(D)^\inverse$.
\end{cor}
\begin{proof}
By definition, the Cantor-Bendixson rank of $X$ is the least ordinal $\alpha$ such that $\limitp^\alpha(X)=\limitp^{\alpha+1}(X)$. By Theorem \ref{teor:derived-1dim}, when $X=\Max(D)^\inverse$ the latter condition is equivalent to $\Theta_\alpha=\Theta_{\alpha+1}$, and thus $\alpha$ is also the Jaffard degree of $\Theta$.
\end{proof}

\begin{cor}\label{cor:scattered}
Let $D$ be a one-dimensional integral domain, and let $X:=\Max(D)^\inverse$. Then, $D$ is ultimately sharp if and only if $X$ is a scattered space.
\end{cor}
\begin{proof}
By Theorem \ref{teor:derived-1dim}, $\limitp^\alpha(X)$ becomes empty if and only if there is an $\alpha$ such that $MT_\alpha=T_\alpha$ for all maximal ideal $M$, where $T_\alpha$ is the $\alpha$-th element of the derived sequence of $\Theta$. However, the latter condition is equivalent to $T_\alpha=K$, i.e., to the fact that $D$ is ultimately sharp.
\end{proof}

When $D$ is a Pr\"ufer domain, a similar construction has been given in \cite[Section 6]{HK-Olb-Re}, following ideas introduced in \cite{loper-lucas-factoring-AD}. In this case, a maximal ideal $M$ is said to be \emph{sharp} if $\bigcap\{D_N\mid N\in\Max(D),N\neq M\}\nsubseteq D_M$, while \emph{dull} otherwise; by \cite[Lemma 6.3(2)]{HK-Olb-Re} and Proposition \ref{prop:jaffov-isolated-dim1} (or by direct proof) we have that $M$ is sharp if and only if $D_M$ is a Jaffard overring of $D$. If $\Max_\sharp(D)$ is the set of sharp maximal ideals of $D$, they define recursively an ascending sequence of rings by
\begin{equation*}
D_\alpha:=\begin{cases}
\bigcap\{(D_\gamma)_M\mid M\in\Max_\sharp(D_\gamma)\} & \text{if~}\alpha=\gamma+1\text{~is a successor ordinal,}\\
\bigcup\{D_\gamma\mid \gamma<\alpha\} & \text{if~}\alpha\text{~is a limit ordinal},
\end{cases}
\end{equation*}
and they show \cite[Lemma 6.5(2)]{HK-Olb-Re} that $\Max(D_\alpha)=\{MD_\alpha\mid M\in\limitp^\alpha(\Max(D)^\inverse)\}$, so that $D_\alpha$ actually coincides with our $T_\alpha$, the $\alpha$-th element of the derived sequence of $\Theta:=\{D_M\mid M\in\Max(D)\}$ (this also, \emph{a fortiori}, for limit ordinals $\alpha$, for which the definitions of $T_\alpha$ and $D_\alpha$ do not coincide). Then, they say that $D$ has \emph{sharp degree} $\alpha$ if $D_\alpha\neq K$ while $D_{\alpha+1}=K$, and that $D$ has \emph{dull degree} $\alpha$ if $D_\alpha=D_{\alpha+1}\neq K$ and $D_\beta\neq D_\alpha$ for all $\beta<\alpha$.

In the case of dull degree, our definition agrees with theirs: it is straightforward to see (using $D_\alpha=T_\alpha$) that $D$ has a dull degree if and only if $D$ is ultimately dull, and that the dull degree of $D$ coincides with the Jaffard degree of $\Theta$. 

On the other hand, for sharp degree, there is a difference: indeed, if $D$ has sharp degree $\alpha$ then $\Theta$ has Jaffard degree $\alpha+1$, and conversely if the Jaffard degree of $\Theta$ is a successor ordinal $\alpha+1$ then $D$ has sharp degree $\alpha$. However, if the Jaffard degree of $\Theta$ is a limit ordinal $\alpha$, then the sharp degree of $D$ does not exist, because the definition requires that the first $\beta$ such that $D_\beta=K$ is a successor ordinal. Thus, $D$ has a sharp degree if and only if $D$ is ultimately sharp and the Jaffard degree of $\Theta$ is a successor ordinal.

Our approach allows to give a simpler form to some of their results. Indeed, Corollary \ref{cor:scattered} above is a more symmetric version of \cite[Theorem 6.6]{HK-Olb-Re}, because it gives a complete equivalence between a topological fact ($\Max(D)^\inverse$ is scattered) and the sharpness of $D$, without requiring that the Jacobson radical of $D$ is nonzero (as in part (2) of the reference) and we do not need to separate the dull and the sharp case (as in parts (1) and (3)).

To conclude the paper, we apply the results of Section \ref{sect:stable} to one-dimensional domain.
\begin{prop}
Let $D$ be a one-dimensional domain. Let $T_\alpha$ be the dull limit of $\Theta:=\{D_M\mid M\in\Max(D)\}$. Then, the family
\begin{equation*}
\Theta':=\{D_M\mid M\in\Max(D),MT_\infty=T_\alpha\}\cup\{T_\alpha\}
\end{equation*}
is stable-preserving.
\end{prop}
\begin{proof}
The claim is a direct consequence of Theorem \ref{teor:stagJ-stable}.
\end{proof}

\begin{prop}\label{prop:ultsharp}
Let $D$ be a one-dimensional Pr\"ufer domain. Then, the family $\Theta:=\{D_M\mid M\in\Max(D)\}$ is stable-preserving if and only if $D$ is ultimately sharp.
\end{prop}
\begin{proof}
If $D$ is ultimately sharp, then its dull limit $T_\alpha$ is equal to $K$, and thus the family $\Theta'=\{D_M\mid M\in\Max(D),MT_\infty=T_\alpha\}\cup\{T_\alpha\}$ of $\Theta$ coincides with $\Theta\cup\{K\}$. Hence, $\Theta\cup\{K\}$ is stable-preserving and so is $\Theta$.

Suppose $D$ is ultimately dull, and let $T:=T_\alpha$ be its dull limit. Consider the set $\Lambda:=\{N\mid NT\neq T\}$. For every $N\in\Lambda$, $D_N=T_{NT}$ is not a Jaffard overring of $T$ (by construction of $T_\alpha$), and thus $\Max(T)\setminus\{NT\}$ is not compact; hence, $\bigcap_{N'\neq N}T_{N'}=T$. For every $N\in\Lambda$, let $\star_N$ be the stable semistar operation
\begin{equation*}
\star_N:I\mapsto\bigcap_{\substack{N'\in\Lambda\\ N'\neq N}}ID_N.
\end{equation*}
Then, $\star_N$ fixes $T$, but $I^{\star_N}=K$ for all $I$ that are $D_N$-modules or $D_M$-modules for some $M\notin\Lambda$. Let $\star$ be the supremum of all the $\star_N$; then, $\star$ fixes $T$. However, $(TD_M)^{\star_N}=K$ for every $M\in\Max(D)$, and thus
\begin{equation*}
K=\bigcap_{M\in\Max(D)}(TD_M)^{\star_N}\neq T^{\star_N}.
\end{equation*}
Hence, $\Theta$ is not stable-preserving.
\end{proof}

\begin{cor}
Let $D$ be an almost Dedekind domain that is ultimately sharp. Then, there is a natural bijection between $\insstable(D)$ and the power set of $\Max(D)$.
\end{cor}
\begin{proof}
By Propositions \ref{prop:stablepres} and \ref{prop:ultsharp}, $\insstable(D)$ is isomorphic to the product $\prod\{\insstable(D_M)\mid M\in\Max(D)\}$. However, each $D_M$ is a discrete valuation ring, and thus $\insstable(D_M)$ contains exactly two operations (the identity and the one sending everything to $K$). The claim follows.
\end{proof}

\bibliographystyle{plain}
\bibliography{/bib/articoli,/bib/libri,/bib/miei}

\begin{thebibliography}{10}

\bibitem{anderson_two_2000}
D.~D. Anderson and Sylvia~J. Cook.
\newblock Two star-operations and their induced lattices.
\newblock {\em Comm. Algebra}, 28(5):2461--2475, 2000.

\bibitem{spectralspaces-libro}
Max Dickmann, Niels Schwartz, and Marcus Tressl.
\newblock {\em Spectral spaces}, volume~35 of {\em New Mathematical
  Monographs}.
\newblock Cambridge University Press, Cambridge, 2019.

\bibitem{dobbs_fedder_fontana}
David~E. Dobbs, Richard Fedder, and Marco Fontana.
\newblock Abstract {R}iemann surfaces of integral domains and spectral spaces.
\newblock {\em Ann. Mat. Pura Appl. (4)}, 148:101--115, 1987.

\bibitem{finocchiaro-ultrafiltri}
Carmelo~A. Finocchiaro.
\newblock Spectral spaces and ultrafilters.
\newblock {\em Comm. Algebra}, 42(4):1496--1508, 2014.

\bibitem{compact-intersections}
Carmelo~A. Finocchiaro and Dario Spirito.
\newblock Topology, intersections and flat modules.
\newblock {\em Proc. Amer. Math. Soc.}, 144(10):4125--4133, 2016.

\bibitem{fontana_factoring}
Marco Fontana, Evan Houston, and Thomas Lucas.
\newblock {\em Factoring {I}deals in {I}ntegral {D}omains}, volume~14 of {\em
  Lecture Notes of the Unione Matematica Italiana}.
\newblock Springer, Heidelberg; UMI, Bologna, 2013.

\bibitem{overrings-prufer-II}
Robert~W. Gilmer, Jr. and William~J. Heinzer.
\newblock Overrings of {P}r\"{u}fer domains. {II}.
\newblock {\em J. Algebra}, 7:281--302, 1967.

\bibitem{well-centered}
William Heinzer and Moshe Roitman.
\newblock Well-centered overrings of an integral domain.
\newblock {\em J. Algebra}, 272(2):435--455, 2004.

\bibitem{HK-Olb-Re}
Olivier~A. Heubo-Kwegna, Bruce Olberding, and Andreas Reinhart.
\newblock Group-theoretic and topological invariants of completely integrally
  closed {P}r\"{u}fer domains.
\newblock {\em J. Pure Appl. Algebra}, 220(12):3927--3947, 2016.

\bibitem{hochster_spectral}
Melvin Hochster.
\newblock Prime ideal structure in commutative rings.
\newblock {\em Trans. Amer. Math. Soc.}, 142:43--60, 1969.

\bibitem{lazard_flat}
Daniel Lazard.
\newblock Autour de la platitude.
\newblock {\em Bull. Soc. Math. France}, 97:81--128, 1969.

\bibitem{loper_sequence}
Alan Loper.
\newblock Sequence domains and integer-valued polynomials.
\newblock {\em J. Pure Appl. Algebra}, 119(2):185--210, 1997.

\bibitem{loper-lucas-factoring-AD}
K.~Alan Loper and Thomas~G. Lucas.
\newblock Factoring ideals in almost {D}edekind domains.
\newblock {\em J. Reine Angew. Math.}, 565:61--78, 2003.

\bibitem{richamn_generalized-qr}
Fred Richman.
\newblock Generalized quotient rings.
\newblock {\em Proc. Amer. Math. Soc.}, 16:794--799, 1965.

\bibitem{starloc}
Dario Spirito.
\newblock Jaffard families and localizations of star operations.
\newblock {\em J. Commut. Algebra}, 11(2):265--300, 2019.

\bibitem{localizzazioni}
Dario Spirito.
\newblock Topological properties of localizations, flat overrings and
  sublocalizations.
\newblock {\em J. Pure Appl. Algebra}, 223(3):1322--1336, 2019.

\bibitem{length-funct}
Dario Spirito.
\newblock Decomposition and classification of length functions.
\newblock {\em Forum Math.}, 32(5):1109--1129, 2020.

\bibitem{starloc2}
Dario Spirito.
\newblock The sets of star and semistar operations on semilocal {P}r\"ufer
  domains.
\newblock {\em J. Commut. Algebra}, 12(4), 2020.

\end{thebibliography}
\end{document}